\documentclass{article}[11pt]
\usepackage[utf8]{inputenc}
\usepackage[nofancy,authblk,savepaper]{anishs}

\graphicspath{pictures}
\usepackage{svg}
\svgpath{pictures}

\usepackage{pstricks}

\usepackage{pdftricks}

\begin{psinputs}
\usepackage{pstricks}
\usepackage{pstricks-add}
\usepackage{color}
\usepackage{pstcol}
\usepackage{pst-plot}
\usepackage{pst-tree}
\usepackage{pst-eps}
\usepackage{multido}
\usepackage{pst-node}
\usepackage{pst-eps}
\SpecialCoor

\end{psinputs}

\newcommand{\asnote}[1]{\authnote{Anish}{blue}{{#1}}}
\newcommand{\wpnote}[1]{\authnote{Wes}{red}{{#1}}}

\title{Direct sampling of short paths for contiguous partitioning}
\author[1]{Wesley Pegden\thanks{Research Supported in part by NSF grant DMS1700365}}
\author[2]{Anish Sevekari}
\affil[1,2]{Department of Mathematical Sciences, Carnegie Mellon University}
\affil[1]{wes@math.cmu.edu}
\affil[2]{asevekar@andrew.cmu.edu}

\begin{document}

\maketitle
\begin{abstract}
	In this paper, we provide a family of dynamic programming based algorithms to sample \emph{nearly-shortest self avoiding walks} between two points of the integer lattice $\Z^2$. We show that if the shortest path of between two points has length $n$, then we can sample paths (\emph{self-avoiding-walks}) of length $n+O(n^{1-\delta)})$ in polynomial time.  As an example of an application, we will show that the Glauber dynamics Markov chain for partitions of the \emph{Aztec Diamonds} in $\Z^2$ into two contiguous regions with nearly tight perimeter constraints has exponential mixing time, while the algorithm provided in this paper can be used be used to uniformly (and exactly) sample such partitions efficiently.
\end{abstract}

\section{Introduction}
\label{s:introduction}
Analysis of political redistrictings has created a significant impetus for the problem of random sampling of graph partitions into connected pieces---e.g., into districtings.

The most common approach to this problem in practice is to use a Markov Chain; e.g., Glauber dynamics, or chains based on cutting spanning trees (e.g., \cite{recomb,matttree,mattmultiscale}).  Rigorous understanding of mixing behavior is the exception rather than the rule;  for example, \cite{barpaths} established rapid mixing of a Markov chain for the special case where both partition classes are unions of horizontal bars, which in each case meet a common side.  No rigorous approach is known, for example, which can approximately uniformly sample from contiguous $2$-partitions even of lattice graphs like the $n\times n$ grid in polynomial time

In this paper we consider a direct approach, where instead of leveraging a Markov chain with unknown mixing time to generate approximate uniform samples, we use a dynamic programming algorithm and rejection sampling to exactly sample from self-avoiding walks in the lattice $\Z^2$ (which correspond to partition boundaries) in polynomial expected time.  Counting self-avoiding lattice walks is a significant long-standing challenge; the connective constant---the base of the exponent in the asymptotic formula for the number of such walks---is not even known for $\Z^2$.  But we will be interested in sampling \emph{nearly-shortest} self avoiding walks, motivated by districting constraints which discourage the use of large district perimeters relative to area.  In particular, we will prove:

\begin{theorem}
	\label{t:algorithm_runtime}
	For any $C$ and $\varepsilon>0$ and for any $n_1,n_2,$ and $n=n_1+n_2$, there is a randomized algorithm which runs w.h.p in polynomial time, and produces a uniform sample from the set of self-avoiding walks in $\Z^2$ from $(0,0)$ to $(n_1,n_2)$ of length at most
\[ n+Cn^{1-\varepsilon}. \]
\end{theorem}

\begin{figure}[t]
    \centering
    \includegraphics[width=.3\linewidth]{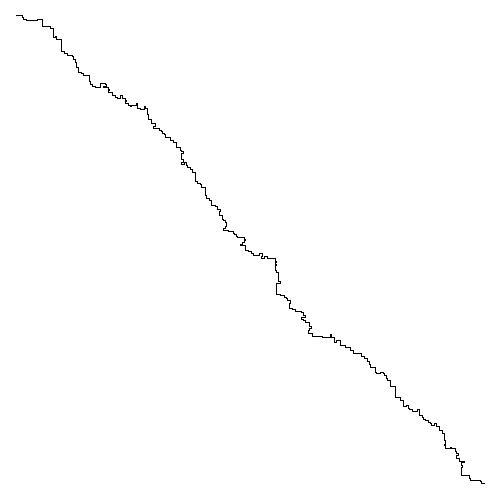}
    \includegraphics[width=.3\linewidth]{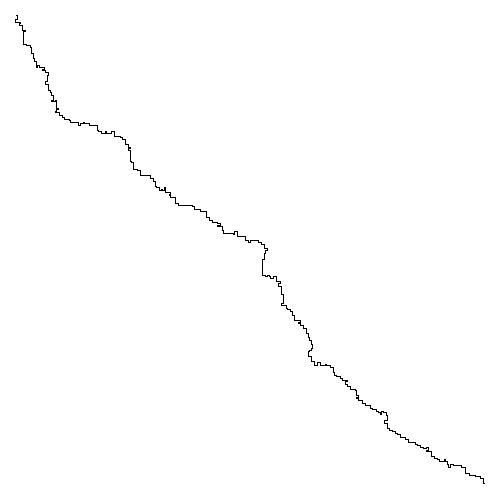}
    \includegraphics[width=.3\linewidth]{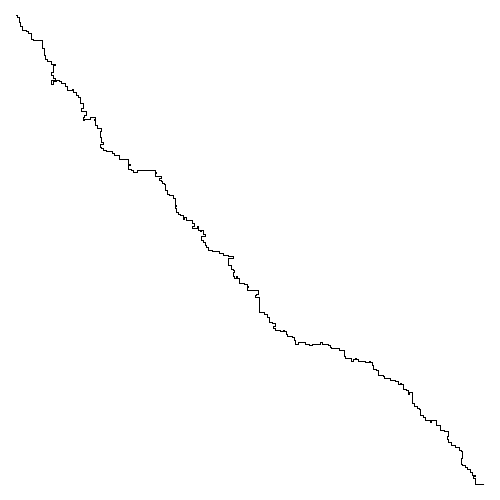}
    \caption{Uniformly random self-avoiding walks of length 700 between corners of a $300\times 300$ grid, generated with the algorithm from \Cref{t:algorithm_runtime}.}
    \label{fig:my_label}
\end{figure}

\begin{figure}[t]
    \centering
    \iftrue
    \includegraphics[width=.3\linewidth]{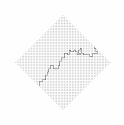}
    \includegraphics[width=.3\linewidth]{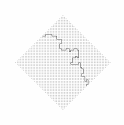}
    \includegraphics[width=.3\linewidth]{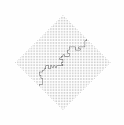}
    \else
    \includegraphics[width=.49\linewidth]{pictures/aztec/1.pdf}
    \includegraphics[width=.49\linewidth]{pictures/aztec/2.pdf}
    \fi
    \caption{Uniformly random self-avoiding walks on $A_{30}$ such that both sides have perimeter of at most $220$.}
    \label{fig:aztec}
\end{figure}

A variant of this algorithm can be used to sample from contiguous $2$-partitions of the Aztec diamond with restricted partition-class perimeter, by sampling short paths between nearly-antipodal points on the dual of the Aztec diamond. These paths are in bijection with the contiguous $2$-partitions of the Aztec diamond, by mapping a partition to it's boundary which gives us a path.
This approach generates samples in polynomial time w.h.p. In contrast, we show that the traditional approach using Markov chains is inefficient:

\begin{theorem}
	\label{t:glauber_mixing}
	For any $C$ and $\varepsilon>0$, Glauber dynamics has exponential mixing time on contiguous $2$-partitions of the Aztec diamond $A_k$ when constrained by perimeter slack $Ck^{1-\varepsilon}$.
\end{theorem}

\noindent \textbf{Organization of the Paper:} The paper is organized in the following manner: \Cref{s:algorithm} describes a dynamic programming algorithm (\Cref{alg:low_girth_walks}) to sample walks without short cycles and proves its correctness. \Cref{s:number_of_paths,s:low_girth_walks} show that the algorithm actually returns a self-avoiding path from $(0,0)$ to $(n_1,n_2)$ in the unbounded lattice graph $\Z^2$ in polynomial time with high probability, enabling the random sampling of paths for rejection sampling. \Cref{s:lattice_subgraphs} provides the same result for \emph{wide} subgraphs of the lattice, the notion of \emph{wide} subgraph is also defined in this section. The last section, \Cref{s:aztec_diamond} is dedicated to proving \Cref{t:glauber_mixing}, and showing that Aztec diamond is a \emph{wide} subgraph of the lattice.

\noindent \textbf{Notation:} For the rest of the paper, we will typically use letters $A,B, \ldots$ for denoting paths from $O = (0,0)$ to $P = (n_1,n_2)$. We will use letters $Q,R, \ldots$ to denote points on the grid.
Each path $A$ from $O$ to $P$ of length $n + 2k$ has two representations, we can describe $A$ by the sequence of moves $a_1, \ldots, a_{n+2k}$ where $a_i \in {L,R,U,D}$ denotes the direction of next step in the path. On the other hand, we can also denote path $A$ by the sequence of points that it visits, namely, $O = A_0, \ldots, A_{n+2k} = P$.
Typically, we will also use $B$ to denote a shortest path, and $A$ to denote a larger path.

We will further let $P_k, W_k, W_k^l$ denote the number of paths (self-avoiding walks), number of walks, and number of walks without cycles smaller than $2l$ from $O$ to $P$ of length $n + 2k$ respectively.

\section{Dynamic Programming Algorithm}
\label{s:algorithm}
In this section, we will describe the dynamic programming algorithms that counts $W_k^l$, the number of walks of length $n+2k$ without short cycles, that is, without cycles of length smaller than $2l$ from $O = (0,0)$ to $P = (n_1,n_2)$ in a subgraph $S$ of the grid $\Z^2$.  The algorithm memorizes the number of paths from every point $Q \in S$ to $P$, along with previous $2l$ steps, which is given by a walk $w$ of length $2l$ ending at $Q$. Let $\Phi_l(Q)$ denote the set of paths ending at $Q$ of length at most $2l$.

\noindent
\hspace{0.01\textwidth}
\begin{minipage}{0.48\textwidth}
	\begin{algorithm}[H]
		\caption{Counting Low Girth Walks}
		\label{alg:low_girth_walks}
		\begin{algorithmic}[1]
			\State $\Call{DP}{Q,P,w,t} = 0$ for $Q \in S$, $w \in \Phi_l(Q)$, $0 \le t \le n + 2k$
			\Function{Walks}{$Q,P,w,t$}
				\If{$t = 0$}
					\If{$Q = P$}
						\State \Return $\Call{DP}{Q,P,w,t} = 1$
					\Else
						\State \Return $\Call{DP}{Q,P,w,t} = 0$
					\EndIf
				\EndIf
				\If{$\Call{DP}{Q,P,w,t} \neq 0$}
					\State \Return $\Call{DP}{Q,P,w,t}$
				\EndIf
				\For{$d \in \set{(1,0),(0,1),(-1,0),(0,-1)}$}
					\If{$Q + d \in S$ \textbf{and} $d \notin w$}
						\State $R = Q + d$
						\State $w'$ is the path obtained by appending $R$ to $w$ and trimming down to length $2l$.
						\State $\Call{DP}{Q,P,w,t} \mathrel{+}= \Call{Walks}{R,P,w',t-1}$
					\EndIf
				\EndFor
				\State \Return $\Call{DP}{Q,P,w,t}$
			\EndFunction
		\end{algorithmic}
	\end{algorithm}
\end{minipage}
\hspace{0.02\textwidth}
\begin{minipage}{0.48\textwidth}
	\begin{algorithm}[H]
		\caption{Sampling Low Girth Walks}
		\label{alg:sampling_walks}
		\begin{algorithmic}[1]
			\Function{Sample Walks}{$k$}
				\State $w = O$
				\For{$i = 0$ to $n + 2k$}
					\For{$Q \sim w[i]$}
					\State $p_{Q} = \Call{DP}{Q, w', n+2k-1-i}$
						\State where $w' = w[i-2l+1] \cdots w[i] Q$ is path of length $2l$ ending at $Q$
					\EndFor
					\State Sample $w[i+1]$ from $Q \sim w[i]$ proportional to $p_Q$.
				\EndFor
				\State \Return $w$
			\EndFunction
		\end{algorithmic}
	\end{algorithm}
	\begin{algorithm}[H]
		\caption{Sampling Paths}
		\label{alg:sampling_paths}
		\begin{algorithmic}[1]
			\Function{Sample Paths}{k}
				\While{$w$ is not a path}
					\State $w = \Call{Sample Walks}{k}$
				\EndWhile
				\State \Return $w$
			\EndFunction
		\end{algorithmic}
	\end{algorithm}
\end{minipage}
\hspace{0.01\textwidth}

Once we have number of these paths, we can sample a walk of length $n+2k$ without cycles of length smaller than $2l$ by starting at $O$ and sampling points in the walk with correct probability using memoized values obtained by \cref{alg:low_girth_walks}.

Since there are at most $4^{2l}$ paths of length $2l$, $\abs{\Phi_l(Q)} \le \sum_{i=0}^{l} 16^i = 2 \cdot 16^{l}$ for any point $Q$. Therefore, size of the DP table in \Cref{alg:low_girth_walks} is $\abs{S} \cdot 16^l$, and each entry in this table takes $O(l)$ time to compute, since deg of each vertex in $S$ is at most $4$. Therefore, \Cref{alg:low_girth_walks} takes $O\brac{\abs{S} \cdot l \cdot 16^l} = O(\abs{S})$ time for constant $l$. Note that these paths are restricted to the set of points $\mathcal{R} = \set{Q \st O - (k,k) \le Q \le P + (k,k)}$. Thus, for large $S$  (in particular for $S = \Z^2$), we can restrict the algorithm to $S' = \mathcal{R} \cap S$. 

Further, once the DP table is computed, \Cref{alg:sampling_walks} runs in $O(n+2k)$ time. We will prove in \Cref{t:long_cycle_bound} that for $k \le Cn^{1 - \varepsilon}$ and $S = \Z^2$, \Cref{alg:sampling_walks} actually returns a path with probability $1 - o(1)$ for $l > \tfrac{1}{\varepsilon}$. This implies that \Cref{alg:sampling_paths} runs in $O(n+2k)$ time with high probability, completing the proof of \Cref{t:algorithm_runtime}. We will provide a sufficient condition for subgraphs $S \subseteq \Z^2$ in \Cref{t:long_cycle_bound_subgraph} which implies the same probability bound for these specific subgraphs $S$.

\section{Number of Paths in a Grid}
\label{s:number_of_paths}
This section focuses on getting bounds on the number of paths from $O = (0,0)$ to $P = (n_1, n_2)$ in the grid. Recall that paths are in fact \emph{self-avoiding walks}. Let $n = n_1 + n_2$ be the length of a shortest path from $O$ to $P$. We will provide some upper and lower bounds on the number of paths of length $n + 2k$ from $O$ to $P$ in terms of number of shortest paths from $O$ to $P$. These upper and lower bounds are based on constructing extensions of shortest paths.

In general, we will associate a shortest \emph{base path} to every path from $O$ to $P$. This association is described in \Cref{d:path_base_map}. We will also provide procedures for extending shortest paths to larger paths, which respects the base path mapping. Then the lower bound on paths of length will follow by bounding the number of extensions of each shortest path, and upper bound will follow from bounding the number of paths of length $n+2k$ that have a specific given path as  the associated \emph{base path}.

Let a shortest path $B$ be described by sequence of moves $b_1, \ldots, b_n$ where $n = n_1 + n_2$, where each $b_i \in \set{U,R}$ describes the direction of move at $i^{\text{th}}$ step. Then we have the following procedure to extend the path $B$ to a path $A$ from $O$ to $P$ of length $n + 2k$.

\begin{definition}
	\label{d:extending_paths}
	Given a shortest path $B$ represented by $b_1, \ldots, b_n$ from $O = (0,0)$ to $P = (n_1,n_2)$ where $n=n_1+n_2$, and a set $M = \set{i_1, \ldots, i_k}$ of indices, we define the extended path $A = \mathcal{A}(B,M)$ obtained by performing following replacements for all $j = 1, \ldots, k$:
	\begin{enumerate}[nosep]
		\item If $b_{i_j} = R$, replace it by $DRU$.
		\item If $b_{i_j} = U$, replace it by $LUR$.
	\end{enumerate}
	For an edge $b_i$, we will also refer to the operation above as \emph{bumping the edge}. Further, we will say that an edge $b_i$ can be \emph{bumped} if bumping the edge $b_i$ gives us a path.
\end{definition}

\begin{figure}
	\centering
	\includesvg[width=0.5\textwidth]{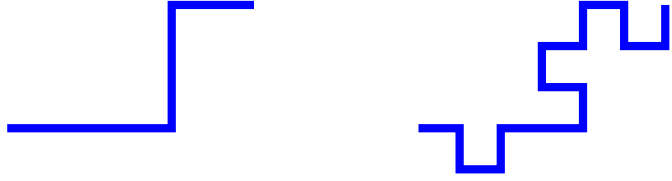}
	\caption{Bumping a shortest path at indices $2,6,9$}
	\label{f:extension_basic_shortest}
\end{figure}

\Cref{f:extension_basic_shortest} illustrates how \Cref{d:extending_paths} behaves when extending shortest paths.
It is not true that for all choices of $M$ the map $\mathcal{A}(B,M)$ is a path. But, we will show that for a large choice of set $M$, it is a path.

\begin{lemma}
	\label{l:extending_paths_correctness}
	For any choice of $M$ such that $b_{i_j - 1} = b_{i_j}$ for all $j$, the map $\mathcal{A}(B,M)$ gives us a path.
\end{lemma}
\begin{proof}
	Let path $B$ be go through the points $O = B_0 \ldots B_n = P$. Then for any point $B_i = (x_i, y_i)$ if the point $X = (x_i - 1, y_i)$ is also in the path $B$ then $X$ must be connected to $B_i$, and hence $B_{i-1} = X$ since otherwise there is a subpath from $(x_i,y_i)$ to $(x_i - 1, y_i + 1)$ (or the other way around) in $B$, which implies that $B$ is not a shortest path.

	In particular, if $b_{i-1} = b_i = U$ then the points to the left of $B_{i-1}$ and $B_i$, that is, the points $(x_{i-1} - 1, y_{i-1})$ and $(x_i - 1, y_i)$ are not in $B$. Therefore, if we replace $b_i$ by $LUR$, we change the portion of path from $B_{i-1}$ to $B_i$ to look like
	\[ B_{i-1} = (x_{i-1}, y_{i-1}) \rightarrow (x_{i-1} - 1, y_{i-1}) \rightarrow (x_{i-1} - 1, y_{i-1} + 1) = (x_i - 1, y_i) \rightarrow (x_i, y_i) = B_i \]
	which is a path since newly added points were not in $B$ initially. Similar argument works for $DRU$ modifications. The modifications of type $U \to LUR$ and $R \to DRU$ happen on opposite side of the path $B$, and hence don't intersect. Further, all the modifications of type $U \to LUR$ don't intersect unless they are adjacent to each other. Therefore, if the set $M$ contains non-adjacent indices, then we can perform all the modifications simultaneously without creating any loops. 
	Further, observe that these modifications do not intersect each other if the set $M$ contains non-adjacent indices, and can be performed simultaneously.
\end{proof}

We will use this procedure described in \Cref{d:extending_paths} to generate a family of paths of length $n + 2k$. To ensure that there are a lot of choices for $M$, we need to argue that most shortest paths from $O$ to $P$ have $\tfrac{n}{2} - o(n)$ many places where the hypothesis of \Cref{l:extending_paths_correctness} is satisfied. This is formalized in the next lemma.

\begin{lemma}
	\label{l:straight_path}
	For any point $P = (n_1, n_2)$ with $n_1 + n_2 = n$, a shortest path from $O = (0,0)$ to $P$ drawn uniformly at random has at least $\frac{n}{2} - O(\sqrt{-n \log \varepsilon})$ places with two consecutive moves in the same direction with probability $1 - \varepsilon$. 
\end{lemma}
\wpnote{still need to replcae $\epsilon$'s with $\varepsilon$'s in some places}
\begin{proof}
	Let $B$ be a shortest path from $O$ to $P$. $B$ can be denoted as a sequence of exactly $n_1$ right moves and exactly $n_2$ up moves. Let us denote this path by $b_1, \ldots, b_n$ where $b_i \in {R,U}$. We can draw a path uniformly at random by picking uniformly at random from a bag with $n_1$ $R$ symbols and $n_2$ $U$ symbols without replacement.
	Let $X_i$ be the indicator random variable for the event that $b_i = b_{i+1}$. Now, we first observe that 
	\[ \pr{X_i \st b_1, \ldots, b_{i-1}} = \frac{p(p-1)}{r(r-1)} + \frac{q(q-1)}{r(r-1)} = \frac{p^2 + q^2 - r}{r(r-1)} \ge \frac{1}{2} - \frac{1}{2(r-1)} \]
	where $p$ is number of $U$ symbols left in the bag, $q$ is number of $R$ symbols left in the bag and $r = p + q$. 
	Now, we will show that  $\pr{X_i \st X_1, \ldots, X_{i-1}} \ge \frac{1}{2} - \frac{1}{n- i - 1}$. It suffices to show that $\pr{X_i = 1\st b_1, \ldots, b_{i-2}, X_{i-1}} \ge \frac{1}{2} - \frac{1}{n - i - 1}$. We will show this by doing two cases: $X_{i-1} = 0$ and $X_{i-1} = 1$. In the first case, $X_{i-1} = 0$,
	\begin{align*}
		\pr{X_i = 1 \st b_1, \ldots, b_{i-2}, X_{i-1} = 0}
		& = \frac{\pr{X_i = 1, X_{i-1} = 0 \st b_1, \ldots, b_{i-2}}}{\pr{X_{i-1} = 0 \st b_1, \ldots, b_{i-2}}} \\
		& = \frac{\frac{p q (q-1) + q p (p-1)}{r(r-1)(r-2)}}{\frac{pq + qp}{r(r-1)}} \\
		& = \frac{pq(p + q - 2)}{2pq(r-2)} = \frac{1}{2} 
	\end{align*}
	where $p$ is number of $U$ symbols left, $q$ is the symbol of $R$ symbol left, and $r = p + q$.
	In the second case, using the same notation, we have
	\begin{align*}
		\pr{X_i = 1 \st b_1, \ldots, b_{i-2}, X_{i-1} = 1}
		& = \frac{\pr{X_i = 1, X_{i-1} = 1 \st b_1, \ldots, b_{i-2}}}{\pr{X_{i-1} = 1 \st b_1, \ldots, b_{i-2}}} \\
		& = \frac{\frac{p(p-1)(p-2) + q(q-1)(q-2)}{r(r-1)(r-2)}}{\frac{p(p-1) + q(q-1)}{r(r-1)}} \\
		& = \frac{p(p-1)(p-2) + q(q-1)(q-2)}{(p(p-1) + q(q-1))(r-2)} \\
		& = \frac{p^3 + q^3 - 3(p^2 + q^2) + 2 (p+q)}{(p^2 + q^2 - (p + q))(r-2)} \\
		& = \frac{r^3 - 3 pqr - 3(r^2 - 2pq) + 2r}{(r^2 - 2pq - r)(r-2)}\\
		& = \frac{r^3 - 3r^2 + 2r - 3pq(r-2)}{(r^2 - r - 2pq)(r-2)}\\
		& = \frac{r(r-1)(r-2) - 3pq(r-2)}{(r^2 - r - 2pq)(r-2)} \\
		& = \frac{r(r-1) - 3pq}{r(r-1) - 2pq}
	\end{align*}
	Note that this term is maximized when $pq$ is minimized, and is minimized when $pq$ is maximized. Constrained to the fact that $p+q = r$ and $p,q \ge 0$, we get
	\[ 1 \ge \frac{r(r-1) - 3pq}{r(r-1) - 2pq} \ge \frac{4r^2 - 4r - 3r^2}{4r^2 - 4r - 2r^2} = \frac{r-4}{2(r-2)} = \frac{1}{2} - \frac{1}{r-2}\]
	Therefore, in both cases, we have
	\[ \pr{X_i = 1 \st X_1, \ldots, X_{i-1}} \ge \frac{1}{2} - \frac{1}{n-i-1} \]
	Now, we couple variables $X_i$ with variables $e_i$, drawn independently such that $\pr{e_i = 1} = \frac{1}{2} - \frac{1}{n-i-1}$. To begin with, we draw $b_1$ with correct probabilities. Then for each $i$, we draw $f_i$ uniformly at random from $[0,1]$. We set $e_i = 1$ if  $f_i \le \pr{e_i = 1}$ and we set $e_i = 0$ otherwise. Further, if $f_i \le \pr{X_i = 1 \st X_1,X_2,\dots,X_{i-1}}$, then we set $a_{i+1}$ such that $X_i = 1$, otherwise we set $a_{i+1}$ such that $X_i = 0$; note that the status of $X_{i+1}$ uniquely determines the choice of $a_{i+1}$. Therefore, $e_i = 1 \implies X_i = 1$, and hence $\sum_{i=1}^{n-1} e_i \le \sum_{i=1}^{n-1} X_i$. Notice that $e_i$ are still independent random variables. Therefore,
	\[ \pr{\sum X_i \le \ex{\sum e_i} - t} \le \pr{\sum e_i \le \ex{\sum e_i} - t} \le \exp\brac{-\frac{2t^2}{n}} \]
	Where the last inequality follows from Hoeffding's inequality. Note that
	\[ \ex{\sum e_i} = \sum \frac{1}{2} - \frac{1}{n-i+1} \ge \frac{n}{2} - 2 \log n \]
	Given any $\varepsilon > 0$, and $t = \sqrt{- n \log \varepsilon}$, we get that 
	\[ \pr{\sum X_i \le \frac{n}{2} - 2 \log n - \sqrt{-n \log \varepsilon}} \le \varepsilon \]
	This proves the required result.
\end{proof}

This allows us to lower bound the number of paths of length $n + 2k$ from $O = (0,0)$ to $P = (n_1, n_2)$ where $n_1 + n_2 = n$. Recall that $P_k$ denotes the number of these paths.

\begin{lemma}
	\label{l:path_extension_basic}
	For any $k \le 0.1 n$ and $1 > \varepsilon \ge 0$, we have the lower bound
	\[ P_k \ge (1 - \varepsilon) P_0 \binom{t - 2k}{k} \]
	where $t = \tfrac{n}{2} - 2 \log n - \sqrt{n \log (1/\varepsilon)}$. Further, there is $n_0 = n_0(\varepsilon)$, such that for all $n \ge n_0$,
	\begin{equation}
		\label{eq:path_extension_bound_basic}
		P_k \ge (1 - \varepsilon) P_0 \frac{(0.49)^k n^k}{k!} \exp\brac{-O \brac{\frac{k^2}{n}}} 
	\end{equation}
\end{lemma}
\begin{proof}
	Consider a path $B$ of length $n$ from $O$ to $P$. Let $B$ be represented by $b_1, \ldots, b_n$ where $b_i \in \set{R,U}$. Then using \Cref{d:extending_paths,l:extending_paths_correctness}, we can extend $B$ to a path $A = \mathcal{A}(B,M)$ of lenght $n+2k$ if we choose $M$ to be a set such that there are no adjacent indices in $M$ and further, for each $i \in M$, $b_{i-1} = b_i$.
	There are at least 
	\[ t = \frac{n}{2} - 2 \log n - \sqrt{n \log (1/\varepsilon)} \]
	such indices, for at least $(1 - \varepsilon)P_0$ many paths. For each of these paths, we need to choose a set of $k$ non-adjacent indices. This can be done in at least
	\begin{equation}
		\label{eq:non_adjacent_choice}
	  \frac{t(t-3)(t-6)\ldots(t-3(k-1))}{k!} \ge \frac{(t-2k)(t-2k-1) \ldots (t-3k+1)}{k!} = \binom{t-2k}{k}
	\end{equation}
	many ways, since after picking first index, we lost $3$ possible choices for rest of the indices.
	Further, observe that any longer path $A$ that is obtained in this way corresponds to exactly one shortest path $B$. We can find this path $B$ by looking at patterns $LUR$ and $DRU$ and replacing them by $U$ and $R$ respectively. If $M$ is choosen satisfying conditions of \Cref{l:extending_paths_correctness}, then it is clear that every $L$ in the extended path $A$ is followed by $UR$ and every $D$ in $A$ is followed by $RU$. Hence, these replacements can be made unambiguously. Since we can do this for all $(1-\varepsilon)P_0$ paths, we get the lower bound.
	\[ P_k \ge (1-\varepsilon)P_0 \binom{t-2k}{k} \]
	Since $2 \log n + \sqrt{n \log(1/\varepsilon)} = o(n)$, there is $n = n(\varepsilon)$ such that for all $n \ge n(\varepsilon)$, $2 \log n + \sqrt{n \log (1/\varepsilon)} \le 0.01 n$, and hence $t \ge 0.49 n$. This gives us the lower bound
	\[ P_k \ge (1-\varepsilon)P_0 \binom{0.49n-2k}{k} \]
	Using \Cref{eq:binomial_bounds}, we have
	\begin{align*}
		P_k
		& \ge (1-\varepsilon) P_0 \frac{(0.49)^k n^k}{k!} \exp\brac{\frac{-4k^2 - k^2 + k}{0.49 n} - \frac{2k(2k+k)}{0.49 n}} \\
		& \ge (1-\varepsilon) P_0 \frac{(0.49)^k n^k}{k!} \exp\brac{-\frac{25k^2}{n}} \\
		\implies P_k & \ge (1-\varepsilon) P_0 \frac{(0.49)^k n^k}{k!} \exp\brac{-O\brac{\frac{k^2}{n}}} 
	\end{align*}
	completing the proof of the lemma.
\end{proof}

The next task is to extend this result to get similar bounds for extending paths of length $n+2k$ to paths of length $n+2k+2l$. We will prove the following:

\begin{lemma}
	\label{l:path_k_extension}
	For any $k, l \le 0.1 n$ and $1 > \varepsilon \ge 0$, there is $n_0 = n_0(\varepsilon)$ such that for all $n \ge n_0(\varepsilon)$,
	\[ P_{k+l} \ge (1 - \varepsilon) P_k \binom{t - 8k - 3l}{l} \binom{k+l}{l}^{-1} \]
	where $t = \tfrac{n}{2} - 2 \log n - \sqrt{n \log (1/\varepsilon) + 2 k n \log n + 30k^2}$. Further, there is $n_1 = n_1(\varepsilon)$, such that for all $n \ge n_1$,
	\begin{equation}
		\label{eq:path_k_extension_bound}
		P_{k+1} \ge (1 - \varepsilon) P_k \frac{(0.49)^l n^l k!}{(k+l)!} \exp\brac{-O\brac{\frac{k(k+l)}{n}}}
	\end{equation}
\end{lemma}

The outline of proof of this lemma will be similar to \Cref{l:path_extension_basic}. Consider a path $A$ of length $n+2k$ from $O$ to $P$.
We want to show that for a large number of sets $M = \set{i_1, \ldots, i_k}$, we can construct the extended path $C = \mathcal{A}(A,M)$.
To ensure we can find a large number of candidates for $M$, we will associate a shortest path to each path $A$. We define a map $\mathcal{B}$ in \Cref{d:path_base_map} such that $\mathcal{B}(A)$ gives us such a shortest path.
We further associate each the edges of $B = \mathcal{B}(A)$ to some of the edges of $A$, and we call these the \emph{good edges} of $A$ and all other edges of $A$ as \emph{bad edges} of $A$. This mapping is defined in \Cref{d:good_edge_mapping}.
We claim that the set of indices where we cannot do modifications in the extension procedure defined in \Cref{d:extending_paths} corresponds to either a corner of $B$ or a \emph{bad edge} of $A$. Then we can bound the number of corners and bad edges to get the bound required.

We begin the proof begin by defining \emph{lattice boxes} to make notation easier, and then use those to define the map $\mathcal{B}$.

\begin{definition}
	\label{d:lattice_box}
	Given points $P_1, P_2 \in \Z^2$, such that $P_1 \le P_2$, we define the \emph{lattice box} $\mathcal{R}(P_1, P_2)$ with left bottom corner $P_1$ and right top corner $P_2$ to be the rectangle with sides parallel to the axis with $P_1$ and $P_2$ as diagonally opposite corners.  To be precise, 
	\[ \mathcal{R}(P_1, P_2) = \set{x \in \Z^2 \st P_1 \le x \le P_2} \]
	We further define \emph{boundary} of a lattice box (and more generally of any set $S \subseteq \Z^2$) to be the set of vertices $v \in S$ such that $v$ has at least one neighbor outside $S$ in the infinite grid graph.
\end{definition}

\begin{definition}
	\label{d:path_base_map}
	We define the map $\mathcal{B}$ as follows. Consider a path $A$ given by points $O = A_0, \ldots, A_{n+2k} = P$ from $O = (0,0)$ to $P = (n_1,n_2)$ with $n_1,n_2 \ge 0$ and $n = n_1 + n_2$.
	We will build $\mathcal{B}(A) = B$ inductively, starting at $O = (0,0)$. We will do this by constructing a sequence of points $R_i$ which will all lie in the intersection $A \cap B$. Let $R_0 = O$. Suppose we have constructed $R_0, \ldots, R_i$.
	\begin{enumerate}[nosep]
		\item Construct a box $\mathcal{R}_i = \mathcal{R}(R_i,P)$ with $R_i$ as the bottom left corner and $P$ as the top right corner.
		\item Find the next point $R_{i+1}$ on $A$, after $R_i$ such that $R_{i+1} \in \mathcal{R}_i$.
		\item Extend $B$ to $R_{i+1}$ using the shortest path along the boundary of $\mathcal{R}_i$ if $R_{i+1} \neq P$.
		\item If $R_{i+1} = P$, then let $\bar{A}$ be part of $A$ between $R_i=(R_i(x),R_i(y))$ and $P$.
			\begin{itemize}[nosep]
				\item If $\bar{A}$ intersects $y = n_2$ before $x = n_1$, define $\bar{R} = (R_i(x), n_2)$
				\item Otherwise define $\bar{R} = (n_1, R_i(y))$.
			\end{itemize}
			Extend $B$ from $R_i$ to $\bar{R}$ to $P$.
	\end{enumerate}
\end{definition}


\begin{lemma}
	\label{l:path_base_map}
	The map $\mathcal{B}$ in \Cref{d:path_base_map} is well defined.
\end{lemma}
\begin{proof}
	Given $R_i \neq P$, we can always find $R_{i+1}$ since $P \in \mathcal{R}_i$ and $P \in A$, so $A$ eventually intersects $\mathcal{R}_i$. Therefore, steps $(1,2)$ in \Cref{d:path_base_map} are well defined. For step $(3)$, observe that $P$ is the only point on boundary of $\mathcal{R}_i$ that has two shortest paths from $R_i$ along the boundary. Therefore, $(3)$ is well defined as long as $R_{i+1} \neq P$.

	For step $(4)$, observe that if $\mathcal{R}_i$ is degenerate, then there is a unique path from $R_i$ to $P$, and this step is well defined. Suppose $\mathcal{R}_i$ is non-degenerate. That is, $R_i$ and $P$ differ at both $x$ and $y$ coordinates. In this case, $\bar{A}$ cannot intersect both the lines $y = n_2$ and $x = n_1$ simultaneously, and it must intersect both of them eventually. Hence, step $(4)$ is well defined as well.
\end{proof}

We now define the \emph{good edge mapping}. First, we will start by making a few notational definitions.

\begin{definition}
	\label{d:path_corner}
	Given a path $A$ from $O$ to $P$ with points $O = A_0, \ldots, A_{n+2k} = P$, we can represent it as a sequence of moves, $a_1 \ldots a_{n+2k}$, where each move is one of the four directions ($U,D,L,R$). We say that $i^{\text{th}}$ point ($A_i$) on this path is a \emph{corner} if $a_i \neq a_{i+1}$. We further include $O$ and $P$ to be corner points.

	We define \emph{last corner point} to be the corner point $Q \neq P$ with highest index. We will also refer to $O$ as the \emph{starting point} and to $P$ as the \emph{ending point}.
\end{definition}

\begin{definition}
	\label{d:path_edges_types}
	Let $A$ be a path of length $n + 2k$ from $O = (0,0)$ to $P = (n_1, n_2)$, where $n = n_1 + n_2$. Let $A$ be given by point $O = A_0, \ldots, A_{n+2k} = P$. Then, we divide edges of $A$ into two categories. Any edge going in the directions $D$ or $L$ will be reffered to as a \emph{reverse edge}, and any edge going in the direction $U$ and $R$ will be reffered to as a \emph{forward edge}.
\end{definition}

\begin{definition}
	\label{d:good_edge_mapping}
	In the setting described in the previous definition, let $B = \mathcal{B}(A)$, where $\mathcal{B}$ is defined in \Cref{d:path_base_map}.
	Let $B$ be given by $O = B_1, \ldots, B_n = P$. We define a \emph{good edge mapping} to be any function $\mathcal{F}_A: \Z_{[0,n-1]} \to \Z_{[0,n+2k-1]}$, where $\Z_{[0,t]} = \Z \cap [0,t]$ satisfying
	\begin{enumerate}[nosep]
		\item $\mathcal{F}_A$ is injective.
		\item For $i < j$, $\mathcal{F}_A(i) < \mathcal{F}_A(j)$.
		\item The edges $A_{\mathcal{F}_A(i)} A_{\mathcal{F}_A(i) + 1}$ and $B_i B_{i+1}$ are \emph{super-parallel}, that is 
			\begin{itemize}[nosep]
				\item If edge $B_i B_{i+1} = (x,y) \to (x,y+1)$, the edge $A_{\mathcal{F}_A(i)} A_{\mathcal{F}_A(i) + 1} = (\bar{x},y) \to (\bar{x},y+1)$ for some $\bar{x}$.
				\item If edge $B_i B_{i+1} = (x,y) \to (x+1,y)$, the edge $A_{\mathcal{F}_A(i)} A_{\mathcal{F}_A(i) + 1} = (x,\bar{y}) \to (x+1,\bar{y})$ for some $\bar{y}$.
			\end{itemize}
	\end{enumerate}
	Given such a mapping $\mathcal{F}$, we will refer to any edge of form $A_{\mathcal{F}(i)} A_{\mathcal{F}(i) + 1}$ to be a \emph{good forward edge}, and any edge that is not a good forward edge as a \emph{bad forward edge}.
\end{definition}

\begin{figure}
	\centering
	\includesvg[width=\textwidth]{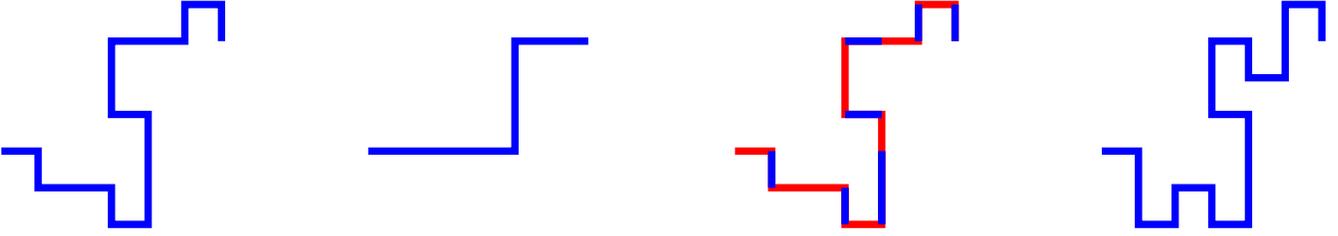}
	\caption{Illustrations for \Cref{d:path_base_map,d:good_edge_mapping}. First image shows a non-shortest path $A$, second image is the \emph{base path} $\mathcal{B}(A)$, third image indicates good forward edges in green, and forth image is the path obtained by bumping at indices $3,14$.}
	\label{f:extension_basic_longer}
\end{figure}

\Cref{f:extension_basic_longer} illustrates the definitions above. We show that such a mapping exists in the lemma below.
\begin{lemma}
	\label{l:good_edge_mapping}
	Given a map $A$ of length $n + 2k$ and let $B = \mathcal{B}(A)$. Using notation in \Cref{d:path_edges_types,d:good_edge_mapping}, there exists a good edge mapping $\mathcal{F}$ satisfying conditions in \Cref{d:good_edge_mapping}.
\end{lemma}
\begin{proof}
  First, it immediately follows from \cref{d:path_base_map,d:path_corner} that all the corners of path $B$ are contained in the set $\set{O = R_0, R_1, \ldots, R_m = P, \bar{R}}$ \wpnote{notaton issue? we are re-using $A$? throughout this proof} \asnote{replacing $R_A$ by $R_m$}, since the portions of $B$ in between these points are straight lines. Now, we define the mapping $\mathcal{F} = \mathcal{F}_A$ for parts of $B$ between $R_i$ and $R_{i+1}$ for $0 \le i \le m-2$, for each edge $B_j B_{j+1}$ between $R_i R_{i+1}$ in $B$, we define $\mathcal{F}(j) = k$ to be the least index such that $A_k A_{k+1}$ and $B_j B_{j+1}$ are super-parallel, that is, they satisfy the condition $(3)$ in \Cref{d:good_edge_mapping}.

	We claim that this is strictly monotonic for each $i$. Suppose not, then there is an index $j$ such that such that $\mathcal{F}(j+1) \le \mathcal{F}(j)$. If $\mathcal{F}(j+1) = \mathcal{F}(j)$, then edges $B_jB_{j+1}$ and $B_{j+1}B_{j+2}$ are super-parallel, which is a contradiction.
	Without loss of generality, let the points $R_i, B_j, B_{j+1}, R_{i+1}$ share the same $x$ coordinate, that is, let $R_i = (x_0,y_0)$, $B_j = (x_0,y_1)$, $B_{j+1} = (x_0,y_1+1)$ and $R_{i+1} = (x_0, y_2)$. Then $A_{\mathcal{F}(j+1)} = (x_1, y_1+1)$ for some $x_1$. Then the path from $R_i = (x_0,y_0)$ to $(x_1, y_1+1)$ must have an edge of the form $(x_2, y_1) \to (x_2,y_1+1)$ since $y_0 \le y_1$. Therefore, there is an index $k < \mathcal{F}(j+1)$ such that $A_k A_{k+1}$ is super-parallel to the edge $B_j B_{j+1}$, which implies $\mathcal{F}(j) < \mathcal{F}(j+1)$, a contradiction!

	If the path between $R_{m-1}$ and $R_{m} = P$ is straight line, we can extend the definition above when $i = A-1$. Otherwise, the point $\bar{R}$ is well defined. Let $\bar{A}$ be portion of $A$ between $R_{m-1}$ and $P$. Without loss of generality, let $\bar{A}$ intersect the line $y = n_2$ before the line $x = n_1$ at a point $Q$.
	Suppose $Q = (x_0,n_2)$, then $x_0 < n_1$, otherwise the path from $R_{m-1}$ to $Q$ will intersect the line $x = n_1$. Since $Q$ is also outside $\mathcal{R}(R_{m-1},P)$, it follows that $x_0 < x_1$ where $R_i = (x_1,y_1)$.

	Now, for all $B_j$ between $R_{m-1}$ and $\bar{R}$, we define $\mathcal{F}(j) = k$ where $k$ is the smallest index such that $A_k$ is between $R_{m_1}$ and $Q$ such that $B_jB_{j+1}$ and $A_kA_{k+1}$ are super-parallel and for all $B_j$ between $\bar{R}$ and $P$, we define $\mathcal{F}(j) = k$ where $k$ is the smallest index such that $A_k$ is between $Q$ and $P$ such that $B_jB_{j+1}$ and $A_kA_{k+1}$ are super-parallel.

	This map is well defined and monotonic since $\bar{A}$ must go from $y = y_1$ to $y = n_2$, and then from $x = x_0$ to $x = n_1$, and hence edges super parallel to $B_jB_{j+1}$ exists for all $B_j$ between $R_{m-1}$ and $P$. Further, the map is strictly monotonic by an argument earlier in the proof. This gives us the \emph{good edge mapping} that we want.
\end{proof}

The next lemma proves that a large number of \emph{good edges} can be bumped.

\begin{lemma}
	\label{l:extending_paths_k_correctness}
	Consider a path $A$ of length $n + 2k$. Let $B = \mathcal{B}$ be the base path associated with it. Suppose $B$ has $c$ corners. Then there is a set $G$ of indices of at least $n - c - 8k$ good edges in $A$ which can be \emph{bumped}.
\end{lemma}
\begin{proof}
	Note that $A$ has exactly $n$ good forward edges, $k$ bad forward edges and $k$ reverse edges. Now, we transverse $A$, and for each good forward edge, we check if we can \emph{bump} the good forward edge.
	To be presice, consider a good forward edge $S_1S_2$. Without loss of generality, we will assume that the edge goes in $U$ direction, and is given by $(x_0, y_0) \to (x_0,y_0 + 1)$.

	Suppose $S_1S_2$ is a good forward edge that cannot be bumped. We will associate either
	\begin{enumerate}[nosep]
	  \item a reverse edge
		\item a bad forward edge
		\item or a corner of $B$
	\end{enumerate}
	as the reason why bumping at $S_1$ is blocked.
	Since $S_1S_2$ cannot be bumped, either $S_3 = (x_0 - 1, y_0)$ is in $A$ or $S_4 = (x_0 - 1, y_0 + 1)$ is in $A$. 

	First, consider the case when $S_3$ is contained in $A$. Look at the edge $e$ going out of $S_3$ in $A$. We have following cases:
	\begin{enumerate}[nosep]
		\item If there is no such edge, then $S_3 = P$. In this case, we say that $P$ blocks bumping at $S_1$.
		\item If the edge $e$ is either a reverse edge or a bad forward edge, then we say that this edge blocks bumping at $S_1$.
		\item If the edge $e$ is going in $U$ direction and is a good forward edge, then there is an unique edge $f \in B$ that is obtained by moving $e$ and $S_1S_2$ perpendicular to their respective directions. This contradicts the definition of $\mathcal{F}$.
		\item If the edge is going in $R$ direction and is a good forward edge, $S_3 S_1 S_2$ are consecutive in $A$. Let $j$ be such that $A_j = S_3, A_{j+1} = S_1$ and $A_{j+2} = S_2$. Since these are good forward edges, there is $i$ such that $\mathcal{F}(i) = j$. Since $\mathcal{F}$ is strictly monotonic, $\mathcal(i+1) = j+1$. Therefore, $B_{i+1}$ is a corner point in $B$. In this case, we say that the corner point $B_{i+1}$ is blocking the bumping at $S_1$.
	\end{enumerate}

	Now, suppose $S_4$ is contained in $A$. Look at the edge $e$ going into $S_4$ in $A$. We again that $4$ cases:
	\begin{enumerate}[nosep]
		\item If there is no such edge, then $S_4 = O$. In this case, we say that $O$ is blocking bumping at $S_1$.
	  \item If the edge $e$ is either a reverse edge or a bad forward edge, then we say that this edge is blocking the bump at $S_1$.
		\item If the edge $e$ is going in $U$ direction and is a good forward edge, then it is exactly the same edge as the one considered in case $(3)$ above.
		\item If the edge $e$ is going in $R$ direction, then both $e$ and $S_1S_2$ end at $S_2$, which cannot happen as $A$ is a path.
	\end{enumerate}

	Each reverse forward edge or backward edge can block at most $4$ good forward edges from bumping, two in each direction, one where it is blocking $S_3$ and one where it is blocking $S_4$. On the other hand, each corner including $O$ and $P$ can block at most one edge. Therefore, there are at least $n - c - 8k$ good forward edges which can be bumped, completing the proof. 
	\asnote{The bound can be improved using a slightly more careful argument while deciding the blocking edges. Some casework can potentially lead to a one to one association between blocking edges and good edges that cannot be bumped.}
\end{proof}

In order to finish the proof of \Cref{l:path_k_extension}, we need a bound on number of paths $A$ of length $n+2k$ such that the base path $B = \mathcal{B}(A)$ has a large number of corners. We will do this by bounding the number of paths $A$ such that $\mathcal{B}(A) = B$, and then using \Cref{l:straight_path} to bound number of paths $B$ with a large number of corners. We will give a rather trivial bound that suffices.

\begin{lemma}
	\label{l:base_path_inverse_bound}
	Given a shortest path $B$ and $k \le 0.1n$, the number of paths $A$ of length $n + 2k$ such that $\mathcal{B}(A) = B$ is at most
	\[ 2 \cdot 3^{2k} \binom{n + 2k}{2k}. \]
\end{lemma}
\begin{proof}
First, we express $B$ as a sequence of directions of length $n$. Now, from $n + 2k$ positions, we choose $2k$ positions, and fill up the rest with the sequence of directions used in $B$. For the remaining $2k$ places, we have at most $3$ choices each since we canot leave in the direction we came from, unless we are picking the starting direction, in which case we might have $4$ choices. This gives an upper bound of
	\[ 3^{2k-1} \brac{4 \binom{n + 2k - 1}{2k - 1} + 3 \binom{n+2k-1}{2k}} = 3^{2k} \binom{n+2k}{2k} + 3^{2k-1} \binom{n+2k-1}{2k-1}\]
	since $\binom{n+2k-1}{2k-1} = \frac{2k}{n + 2k} \binom{n+2k}{2k} \le 3 \binom{n+2k}{2k}$ for $k \le 0.1n$, we get the result.
	\asnote{This bound can potentially be improved, for example, by looking at the number of corners of $B$.}
\end{proof}

Now we are in a position to finish the proof of \Cref{l:path_k_extension}.
\begin{proof}
	Recall that by \Cref{l:path_extension_basic}, there is $n_0 = n_0(\varepsilon)$ such that for all $n \ge n_0$,
	\[ P_k \ge \frac{1}{2} P_0 \frac{(0.49)^k n^k}{k!}\exp\brac{-\frac{25k^2}{n}} \]
	\wpnote{There seem to be some problems with the following sentence and I'm not sure what it is supposed to be claiming (and is it $+\sqrt{n\log(1/\varepsilon_1)}$ or $-\sqrt{n\log(1/\varepsilon_1)}$?)} \asnote{It should be at least} On the other hand, for any given $\varepsilon_1$, we have that the number of paths $A$ such that the base path $B = \mathcal{B}(A)$ has at least $\tfrac{n}{2} + 2 \log n + \sqrt{n \log (1/\varepsilon_1)}$ corners is upper bounded by
	\[ 2 \varepsilon_1 P_0 3^{2k} \binom{n+2k}{2k} \le 2 \varepsilon_1 P_0 \frac{3^{2k}n^{2k}}{(2k)!} \exp\brac{\frac{8k^2 - 4k^2 + 2k}{n}} \le 2 \varepsilon_1 P_0 \frac{3^{2k}n^{2k}}{(2k)!} \exp\brac{\frac{5k^2}{n}} =  T \]
	Hence, if we choose $\varepsilon_1$ such that
	\[ \varepsilon_1 \le \frac{\varepsilon}{4} \cdot \frac{(0.49)^k (2k)!}{3^{2k} n^k k!} \exp\brac{\frac{-30k^2}{n}} \]
	or equivalently, if 
	\[ \log (1/\varepsilon_1) \ge \log (1/\varepsilon) + \log 4 + k \log n + 4k \log 3 - k \log k + \frac{30k^2}{n} \]
	It follows that there are at most $\varepsilon P_k$ paths $A$ of length $n+2k$ such that $B$ has at most 
	\[ \frac{n}{2} + 2 \log n + \sqrt{n \log (1/\varepsilon) + 2nk\log n + 30k^2} \]
	corners, when $k \le 0.1n$ and $n \ge 81$.
	Therefore, in this setting, every path $A$ has at least $t - 8k$ good edges which can be bumped where 
	\[ t = \frac{n}{2} - 2 \log n - \sqrt{n \log (1/\varepsilon) + 2kn \log n + 30k^2} \]
	Note that every edge that is bumped can prevent at most $3$ new edges from being bumped.
	For example, if we bump and edge that looks like $(x_0,y_0) \to (x_0,y_0+1)$ it can stop the edges $(x_0,y_0-1) \to (x_0,y_0)$, $(x_0,y_0+1) \to (x_0,y_0+2)$ and $(x_0-2, y_0+2) \to (x_0-1,y_0+2)$ from bumping, which it initially did not. Therefore, we can choose set $M$ of $l$ edges which can be bumped simultaneously in 
	\[ \frac{t(t-4)(t-8)\cdots(t-4(l-1))}{l!} \ge \frac{(t-3l) \cdots (t-4l+1)}{l!} = \binom{t-3l}{l} \]
	many ways. Further, each path of length $n+2k+2l$ can have $k+l$ bumps, and can potentially be obtained in $\binom{k+l}{k}$ many different paths of length $l$. This gives us the lower bound
	\[ P_{k+l} \ge (1 - \varepsilon) P_k \binom{t - 8k - 3l}{l} \binom{k+l}{k}^{-1} \]
	as required. Note that for $k \le \tfrac{n}{(\log n)^2}$, there exists $n(\varepsilon)$ such that for all $n \ge n(\varepsilon)$, $t \ge 0.49n$.
	Using \Cref{eq:binomial_bounds}, we get the simplified lower bound:
	\begin{align*}
	  P_{k+l}
		& \ge (1-\varepsilon) P_k \frac{(0.49)^l n^l k!}{(k+l)!} \exp\brac{-\frac{2(8k+3l)l - l^2 + l}{0.49n} - \frac{2l(8k+3l)}{0.49n}}\\
		& \ge (1-\varepsilon) P_k \frac{(0.49)^l n^l k!}{(k+l)!} \exp\brac{-\frac{32(kl+l^2)}{0.49n}} \\
		& \ge (1-\varepsilon) P_k \frac{(0.49)^l n^l k!}{(k+l)!} \exp\brac{-\frac{70(kl+l^2)}{n}} \\
		& \ge (1-\varepsilon) P_k \frac{(0.49)^l n^l k!}{(k+l)!} \exp\brac{-O\brac{\frac{(kl+l^2)}{n}}}
	\end{align*}
\end{proof}

\section{Number of Low Girth Walks in the Grid}
\label{s:low_girth_walks}
In this section, we will use the bounds obtained in the section above to compare the number of paths from $O = (0,0)$ to $P = (n_1,n_2)$ to the number of walks from $O$ to $P$ that do not have cycles of length less than $2l$. For the sake of notation, let $W_k^l$ denote the number of walks from $O$ to $P$ that do not have cycles of length less than $2l$. Then we have the following:
\begin{theorem}
	\label{t:long_cycle_bound}
	Given constants $C,\delta,\alpha \ge 0$, there exists $n(C,\delta,\alpha)$, such that for all $n \ge n(C,\delta,\alpha)$, and for all $k,l$ such that $k \le Cn^{1-\delta}$ and $l\delta > 1 + 2\alpha$,
	\begin{equation}
		\label{eq:long_cycle_bound}
	  P_k \le W_k^l \le \brac{1 + 16 n^{-\alpha}} P_k.
	\end{equation}
\end{theorem}
\begin{proof}
	We will show this by induction on $k$. Note that result holds for $0 \le k < l$ since in this setting, $W_k^l = P_k$. Suppose by induction hypothesis, $W_{\bar{k}}^l \le \brac{1 + 8n^{-\alpha}} P_{\bar{k}}$ for $0 \le \bar{k} < k$. Since every walk of length $n+2k$ with no cycles of length smaller than $2l$ is either or path or can be decomposed into a cycle of length $t\geq 2\ell$ and a walk of length $n+2k-t$ with no cycles of length smaller than $2l$, we get the following bound:
	\[ W_k^l \le P_k + \sum_{t=0}^{k-l} W_t^l 16^{k-t} (n+2t) \le P_k \sum_{t=0}^{k-l} \brac{1 + 8n^{-\alpha}} \cdot 2 \cdot P_t 16^{k-t} n.\]
	Here $16^t$ is a simple upper bound on the number of cycles of length $16^t$ through a fixed point. Note that for $t \le Cn^{1-\delta}$, $n+2t \le 2n$. Now, using \Cref{l:path_k_extension} with $\varepsilon = 0.5$, we have
	\begin{align*}
		\frac{P_t 16^{k-t} n}{P_k}
		& \le 2 \cdot \frac{k!}{t!} \cdot \frac{16^{k-t}n}{n^{k-t}(0.49)^{k-t}} \exp\brac{\frac{70(k-t)(k-t+t)}{n}} \\
		& \le 2 \exp\brac{(k-t)(\log k + \log 16 - \log n - \log(0.49)) + \frac{70(k-t)k}{n} + \log n} \\
		& \le 2 \exp\brac{(k-t)((1-\delta) \log n + \log C - \log n + \log 40) + \frac{70k(k-t)}{n} + \log n}.
	\end{align*}
	Let $k-t = l+r$, and let $l$ be an integer constant such that $l\delta > 1$, then we can upper bound the summation as below:
	\begin{align*}
		\frac{W_k^l}{P_k}
		& \le 1 + \sum_{r=0}^{k-l} \brac{1 + 8n^{-\alpha}} \cdot 4 \cdot \exp\brac{(1 - l\delta)\log n + C_1 l + - r \delta \log n + C_1 r + \frac{50k(l+r)}{n}} \\
		& \le 1 + 4 \brac{1 + 8n^{-\alpha}} \exp\brac[\bigg]{(1  - l\delta) \log n + C_1 l + 50lCn^{-\delta}} \brac{\sum_{r=0}^{k-l} \exp\brac{r\brac{-\delta \log n + C_1 + 50Cn^{-\delta}}}} \\
		& \le 1 + 4 \brac{1 + 8n^{-\alpha}} \exp\brac[\bigg]{- \alpha \log n} \brac{\sum_{r=0}^{\infty}\exp \brac{-r C_2 \log n}},
	\end{align*}
	where these equations hold with constants $C_1 = \log 40C$ and $C_2 = \tfrac{\delta}{2}$ for $n \ge n_1(C,\delta)$. Simplifying, we get the upper bound:
	\begin{align*}
		\frac{W_k^l}{P_k}
		& \le 1 + 4 \brac{1 + 8n^{-\alpha}} n^{-\alpha} \frac{1}{1 - n^{-C_2}} \\
		& \le 1 + 16 \brac{n^{-\alpha}},
	\end{align*}
	where the last inequality holds for $n \ge n_2(\alpha)$, so that $8n^{-\alpha}, n^{-C_2} \le 0.5$. Therefore, for $n \ge n(C,\delta,\alpha) = \max(n_1(C,\delta), n_2(\alpha))$, we get the result.
\end{proof}

\section{Subgraphs of the Lattice}
\label{s:lattice_subgraphs}
In this section, we do the same analysis for number of paths in induced subgraphs of the lattice $\Z^2$. To ensure that the sampling procedure works efficiently, we will prove the analogues of \Cref{l:path_extension_basic,l:path_k_extension,t:long_cycle_bound} where we restrict ourselves to paths bounded in some set $S \subseteq \Z^2$. First, let us setup some notation:
\begin{notation*}
	For this section, let $S \subseteq \Z^2$ be an induced subset of lattice. Let $O,P$ be two points in $S$. Without loss of generality, we will assume that $O = (0,0)$ and $P = (n_1,n_2) \ge O$. Let $n = n_1 + n_2$ denote the length of shortest path from $n_1$ to $n_2$ in $\Z^2$.
	Let $P_k$ denote the number of paths (\emph{self avoiding walk}) from $O$ to $P$ of length $n+2k$ that are contained in $S$
	Let $W_k^l$ denote the number of walks from $O$ to $P$ of length $n+2k$ that do not have cycles of length smaller than $l$ and are contained in $S$.
\end{notation*}
Now, we make a few definitions which are helpful in the analysis
\begin{definition}
	\label{d:lattice_boundary}
	Given set $S \subseteq \Z^2$, we define the boundary of $S$, denoted by $\partial S$ as the set of points $Q \in S$ such that at least on neighbor of $Q$ is outside $S$.
\end{definition}
\begin{definition}
	\label{d:wide_subgraph}
	Given an induced subgraph $S \subseteq \Z^2$ and points $O,P \in S$, we say that $S$ is $(k,s,\beta)$-wide if at least $(1-\beta)$ fraction of paths of length $n+2k$ from $O$ to $P$ contained in $S$ intersect the boundary $\partial S$ of $S$ in at most $s$ points.
\end{definition}
To give some trivial examples, every set $S$ is $(k,s,1)$ wide for all $k,s$ and on the other hand, every set $S$ is $(k,n+2k,\beta)$-wide for all $k,\beta$. We are now ready to state and prove variants of \Cref{l:path_extension_basic,l:path_k_extension} that hold for bounded subgraphs of the lattice $\Z^2$.
\begin{lemma}
	\label{l:path_extension_basic_subgraph}
	Given an induced subgraph $S \subseteq \Z^2$ and points $O,P$ in $S$ such that $S$ is $(0,s,\beta)$-wide, and numbers $k \in \Z$ and $\varepsilon \in \R$, $\varepsilon,k > 0$, we have the lower bound on number of paths from $O$ to $P$ contained in $S$:
	\[ P_k \ge (1 - \varepsilon - \beta) \binom{t - 2s - 2k}{k} \]
	where $t = \frac{n}{2} - 2 \log n - \sqrt{n \log (1/\varepsilon)}$. Further, there is $n_0 = n_0(\varepsilon)$ such that for all $n \ge n_0$,
	\begin{equation}
		\label{eq:path_extension_bound_basic_subgraph}
		P_k \ge (1 - \varepsilon-\beta) P_0 \frac{(0.49)^k n^k}{k!} \exp\brac{-O \brac{\frac{k(k+s)}{n}}} 
	\end{equation}
\end{lemma}
\begin{proof}
	The proof is almost the same as \Cref{l:path_extension_basic}, except one major change, we need to ensure that the constructed paths $\mathcal{A}(B,M)$ using \Cref{d:extending_paths} stays inside set $S$.
	We can bump a path $B$ at index $i$ if the point $B_i$ and $B_i+1$ are not on the boundary $\partial S$.
	Further, there are at least $(1 - \varepsilon - \beta)P_0$ shortest paths that have at most $\frac{n}{2} + 2 \log n + \sqrt{n \log (1/\varepsilon)}$ corners and at most $s$ points that are on the boundary.
	For these paths, there are at least $\frac{n}{2} - 2 \log n - \sqrt{n \log (1/\varepsilon)} - 2s$ indices which can be bumped while keeping the path inside set $S$. Using \Cref{eq:non_adjacent_choice}, we get the lower bound:
	\[ P_k \ge ( 1 - \varepsilon - \beta ) \binom{t - 2s - 2k}{k} \]
	for 
	\[ t = \tfrac{n}{2} - 2 \log n - \sqrt{n \log (1 / \varepsilon)}. \]
	Since $t = \tfrac{n}{2} - o(n)$ there is $n_0 = n_0(\varepsilon)$ such that for all $n \ge n_0$, $t \ge 0.49n$. This gives us the lower bound, due to computation similar to \Cref{l:path_extension_basic}.
	\begin{align*}
		P_k
		& \ge (1-\varepsilon-\beta) P_0 \frac{(0.49)^k n^k}{k!} \exp\brac{\frac{-2(2s+2k)k - k^2 + k}{0.49 n} - \frac{2k(2k+k+2s)}{0.49 n}} \\
		& \ge (1-\varepsilon) P_0 \frac{(0.49)^k n^k}{k!} \exp\brac{-\frac{25k(k+s)}{n}} \\
		\implies P_k & \ge (1-\varepsilon) P_0 \frac{(0.49)^k n^k}{k!} \exp\brac{-O\brac{\frac{k(k+s)}{n}}} 
	\end{align*}
	completing the proof of the lemma.
\end{proof}
\begin{lemma}
	\label{l:path_k_extension_subgraph}
	Given an induced subgraph $S \subseteq \Z^2$ and points $O,P$ in $S$ such that $S$ is $(k,s,\beta)$-wide and $(0,s,\beta)$-wide, and numbers $k \in \Z$ and $\varepsilon \in \R$, $\varepsilon,k > 0$, then there is $n_0 = n_0(\varepsilon)$ such that we have the lower bound on number of paths from $O$ to $P$ contained in $S$ for $n \ge n_0(\varepsilon)$:
	\[ P_{k+l} \ge (1 - \varepsilon - \beta) \binom{t - 2s - 8k - 3l}{l} \]
	where $t = \tfrac{n}{2} - 2 \log n - \sqrt{n \log (1/\varepsilon) + 2 k n \log n + 30k^2}$. Further, if $k,s \le \tfrac{n}{(\log n)^2}$, there is $n_1 = n_1(\varepsilon)$ such that for all $n \ge n_1$,
	\begin{equation}
		\label{eq:path_k_extension_bound_subgraph}
		P_{k+1} \ge (1 - \varepsilon - \beta) P_k \frac{(0.49)^l n^l k!}{(k+l)!} \exp\brac{-O\brac{\frac{l(k+s+l)}{n}}}
	\end{equation}
\end{lemma}
\begin{proof}
	The proof of this lemma is similar to \Cref{l:path_k_extension}, and we will only mention the key differenecs. First, observe that if $B = \mathcal{B}(A)$ has $c$ corners, then there are at least $n - c - 8k$ indices in $A$ that can be bumped. Among these, there are at most $2s$ indices where the points $A_i$ or $A_{i+1}$ are on boundary. Further, choice of $\varepsilon_1$ in the proof of \Cref{l:path_k_extension} changes to satsify
	\[ \log (1/\varepsilon_1) \ge \log (1/\varepsilon) + \log 4 + k \log n + 4k \log 3 - k \log k + \frac{30k(k+s)}{n} \]
	Therefore, there are at most $\varepsilon P_k$ paths $A$ of length $n+2k$ such that $B$ has at most 
	\[ \tfrac{n}{2} + 2 \log n + \sqrt{n \log (1/\varepsilon) + 2nk \log n + 30k(k+s)} \]
	corners, there are at most $\beta P_k$ paths $A$ of length $n + 2k$ that may have more that $s$ points on the boundary $\partial S$.
	This gives us that at least $(1 - \varepsilon - \beta)P_k$ paths of length $n+2k$ can be bumped at $t - 2s - 8k$ positions for 
	\[ t = \tfrac{n}{2} - 2 \log n - \sqrt{n \log(1/\varepsilon) + 2nk \log n + 30k(k+s)} \]
	For $k,s \le \tfrac{n}{(\log n)^2}$, $t = \tfrac{n}{2} - o(n)$, implying that there is $n_1 = n_1(\varepsilon)$ such that $t \ge 0.49n$. Using \Cref{eq:binomial_bounds} and computations similar to \Cref{l:path_k_extension}, we get the lower bound:
	\begin{align*}
	  P_{k+l}
		& \ge (1-\varepsilon-\beta) P_k \frac{(0.49)^l n^l k!}{(k+l)!} \exp\brac{-\frac{2(8k+3l+2s)l - l^2 + l}{0.49n} - \frac{2l(8k+3l+2s)}{0.49n}}\\
		& \ge (1-\varepsilon-\beta) P_k \frac{(0.49)^l n^l k!}{(k+l)!} \exp\brac{-\frac{70l(k+l+s)}{n}}
	\end{align*}
	This gives us the proposed bound, finishing the proof.
\end{proof}
Next step is to prove that variant of \Cref{t:long_cycle_bound} holds for induced subgraph $S$ of the lattice provided that the set $S$ is satisfies certain properties.
\begin{theorem}
	\label{t:long_cycle_bound_subgraph}
	Given constants $C,\delta,\alpha \ge 0$, a subgraph $S \subseteq \Z^2$, and a function $s=s(k)$ such that $S$ is $(k,s(k),\beta)$-wide where $\beta \le 0.25$ and $s(k) \le Cn^{(1-\delta)}$ for all $k \le Cn^{(1-\delta)}$, there exists $n_0=n_0(C,\delta,\alpha)$ such that for all $n \ge n_0$, $k \le Cn^{(1-\delta)}$ and $l \ge 0$ such that $l\delta > 1 + 2 \alpha$,
	\begin{equation}
		\label{eq:long_cycle_bound_subgraph}
		P_k \le W_k^l \le (1 + 32n^{-\alpha}) P_k
	\end{equation}
\end{theorem}
\begin{proof}
	The proof is similar to the proof of \Cref{t:long_cycle_bound}. The recursive bound still holds, that is,
	\[ W_k^l \le P_k + \sum_{t=0}^{k-l} W_t^l 16^{k-t} (n+2t) \le P_k \sum_{t=0}^{k-l} \brac{1 + 8n^{-\alpha}} \cdot 2 \cdot P_t 16^{k-t} n \]
	since we are restricting all the paths and walks to be restricted to set $S$. Using \Cref{l:path_k_extension_subgraph} with $\varepsilon = 0.5$, we get
	\begin{multline*}
		\frac{P_t 16^{k-t} n}{P_k}
		 \le 4 \cdot \frac{k!}{t!} \cdot \frac{16^{k-t}n}{n^{k-t}(0.49)^{k-t}} \exp\brac{\frac{70(k-t)(k-t+t+s)}{n}} \\
		 \le 4 \exp\bigg((k-t)((1-\delta) \log n + \log C - \log n 
		  + \log 40) + \frac{70(k+s)(k-t)}{n} + \log n\bigg),
	\end{multline*}
	which follows from computations in \Cref{t:long_cycle_bound}.
	The last expression holds for $C_1 = \log 40C$ and $C_2 = \tfrac{\delta}{2}$ for $n \ge n_1(C,\delta)$. Following the steps in \Cref{t:long_cycle_bound} to evaluate the summation, we get the upper bound
	\[ W_k^l \le 1 + 8 \brac{1 + 32n^{-\alpha}} n^{-\alpha} \frac{1}{1 - n^{-C_2}} \le 1 + 32n^{-\alpha}, \]
	where the last inequality holds for $n \ge n_2(\alpha)$ chosen such that $16n^{-\alpha}, n^{-C_2} \le 0.5$. Therefore, for $n \ge n(C,\delta,\alpha) = \max(n_1(C,\delta),n_2(\alpha))$, we get the result.
\end{proof}

\newcommand{\ZZ}{\mathbf{Z}}
\newcommand{\ZZp}{\mathbf{Z}'}
\newcommand{\PP}{\mathcal{P}}

\section{The Aztec Diamond}
\label{s:aztec_diamond}
We let $\ZZ$ denote the planar graph of the integer lattice $\Z^2$ and let $\ZZp$ be its planar dual, with vertices using half-integer coordinates.

We define the Aztec Diamond graph $A_k$ to be the subgraph of $\ZZp$ induced by the set
\begin{equation}\label{aztec}
    V(A_k)=\{(x,y)\in \Z^2+(\tfrac 1 2,\tfrac 1 2)\mid |x|+|y|\leq k\},
\end{equation}
and define $A_k'$ to be the subgraph of $\ZZ$ induced by the set
\begin{equation}\label{aztecprime}
    V(A_k')=\{(x,y)\in \Z^2\mid |x|+|y|\leq k\};
\end{equation}
see Figure \ref{f.aztec}.  We define the boundary $\partial A_k'$ to be those vertices of $A_k'$ $(x,y)$ with $|x|+|y|=k$.

\begin{figure}
  \begin{pdfpic}
    \psset{unit=.3cm}
    \begin{pspicture}[shift=*](-1,-1)(1,1)
      \psdots(.5,.5)(.5,-.5)(-.5,-.5)(-.5,.5)

      \psdots[dotsize=1.5pt](0,0)
      \psdots[dotsize=1.5pt](0,1)(1,0)(0,-1)(-1,0)
    \end{pspicture}\hspace{1cm}
        \begin{pspicture}[shift=*](-2,-2)(2,2)
          \psdots(.5,.5)(.5,-.5)(-.5,-.5)(-.5,.5)

          \psdots(1.5,.5)(1.5,-.5)(-1.5,.5)(-1.5,-.5)
          \psdots(.5,1.5)(-.5,1.5)(.5,-1.5)(-.5,-1.5)

          \psdots[dotsize=1.5pt](0,0)

          \psdots[dotsize=1.5pt](0,1)(1,0)(0,-1)(-1,0)
          \psdots[dotsize=1.5pt](-1,-1)(1,1)(1,-1)(-1,1)

          \psdots[dotsize=1.5pt](0,2)(2,0)(0,-2)(-2,0)
    \end{pspicture}\hspace{1cm}
        \begin{pspicture}[shift=*](-3,-3)(3,3)
          \psdots(.5,.5)(.5,-.5)(-.5,-.5)(-.5,.5)

          \psdots(1.5,.5)(1.5,-.5)(-1.5,.5)(-1.5,-.5)
          \psdots(.5,1.5)(-.5,1.5)(.5,-1.5)(-.5,-1.5)
          
          \psdots(1.5,1.5)(1.5,-1.5)(-1.5,1.5)(-1.5,-1.5)
          \psdots(2.5,.5)(2.5,-.5)(-2.5,.5)(-2.5,-.5)
          \psdots(.5,2.5)(.5,-2.5)(-.5,2.5)(-.5,-2.5)

          \psdots[dotsize=1.5pt](0,0)

          \psdots[dotsize=1.5pt](0,1)(1,0)(0,-1)(-1,0)
          \psdots[dotsize=1.5pt](-1,-1)(1,1)(1,-1)(-1,1)

          \psdots[dotsize=1.5pt](0,2)(2,0)(0,-2)(-2,0)
          \psdots[dotsize=1.5pt](-2,-1)(2,1)(2,-1)(-2,1)
          \psdots[dotsize=1.5pt](-1,-2)(1,2)(1,-2)(-1,2)

          \psdots[dotsize=1.5pt](0,3)(3,0)(0,-3)(-3,0)
          
        \end{pspicture}
            \end{pspicture}\hspace{1cm}
        \begin{pspicture}[shift=*](-3,-3)(3,3)
          \psdots(.5,.5)(.5,-.5)(-.5,-.5)(-.5,.5)

          \psdots(1.5,.5)(1.5,-.5)(-1.5,.5)(-1.5,-.5)
          \psdots(.5,1.5)(-.5,1.5)(.5,-1.5)(-.5,-1.5)
          
          \psdots(1.5,1.5)(1.5,-1.5)(-1.5,1.5)(-1.5,-1.5)
          \psdots(2.5,.5)(2.5,-.5)(-2.5,.5)(-2.5,-.5)
          \psdots(.5,2.5)(.5,-2.5)(-.5,2.5)(-.5,-2.5)

          \psdots(2.5,1.5)(2.5,-1.5)(-2.5,1.5)(-2.5,-1.5)
          \psdots(1.5,2.5)(1.5,-2.5)(-1.5,2.5)(-1.5,-2.5)
          \psdots(3.5,.5)(3.5,-.5)(-3.5,.5)(-3.5,-.5)
          \psdots(2.5,1.5)(2.5,-1.5)(-2.5,1.5)(-2.5,-1.5)
          \psdots(1.5,2.5)(1.5,-2.5)(-1.5,2.5)(-1.5,-2.5)
          \psdots(.5,3.5)(.5,-3.5)(-.5,3.5)(-.5,-3.5)

          \psdots[dotsize=1.5pt](0,0)

          \psdots[dotsize=1.5pt](0,1)(1,0)(0,-1)(-1,0)
          \psdots[dotsize=1.5pt](-1,-1)(1,1)(1,-1)(-1,1)

          \psdots[dotsize=1.5pt](0,2)(2,0)(0,-2)(-2,0)
          \psdots[dotsize=1.5pt](-2,-1)(2,1)(2,-1)(-2,1)
          \psdots[dotsize=1.5pt](-1,-2)(1,2)(1,-2)(-1,2)
        \psdots[dotsize=1.5pt](-2,-2)(2,2)(2,-2)(-2,2)

          \psdots[dotsize=1.5pt](0,3)(3,0)(0,-3)(-3,0)
        \psdots[dotsize=1.5pt](-3,-1)(3,1)(3,-1)(-3,1)
        \psdots[dotsize=1.5pt](-1,-3)(1,3)(1,-3)(-1,3)
        \psdots[dotsize=1.5pt](-4,0)(0,-4)(4,0)(0,4)

        \psline[linewidth=1.5pt](-2,-2)(-2,0)(-1,0)(-1,1)(2,1)(2,2)
        
        \end{pspicture}
  \end{pdfpic}
  \caption{\label{f.aztec}The large dots show the vertex sets of the the Aztec diamonds $A_1$, $A_2,$ $A_3,$ and $A_4$, which are subsets of the dual lattice $\ZZp$.  The small dots show the vertex sets of the corresponding $A_1',A_2',A_3'$ and $A_4'$, which are subsets of the integer lattice $\ZZ=\Z^2$.  In the last case, a path $\PP_{\omega}$ in $A_4'$ corresponding to a partition $\omega$ of $A_4$ is shown.}
\end{figure}

We consider as a toy example the problem of randomly dividing the Aztec diamond into two contiguous pieces $S_1,S_2$, whose boundaries are both nearly as small as possible.  Here we use the \emph{edge-boundary} of $S_i$, which is the number of edges between $S_i$ and $\ZZp\setminus S_i$.  Note that this is the same has the length of the closed walk in $\ZZ$ enclosing $S_i$.   We collect the following simple observations about these sets and their boundaries:
\begin{observation}\label{Akbound}
  $A_k$ has $8k$ boundary edges.\qed
\end{observation}
\begin{observation}\label{antipodal2k}
  Every shortest path in $A_k'$ between antipodal points on $\partial A_k'$ has length $2k$.\qed
\end{observation}
\begin{observation}\label{aligned}
  For $x\geq 0$, the (unique) shortest path between points $(x,y_1)$ and $(x,-y_1)$ of $\partial A_k'$ has length $2k-2x$.\qed
  \end{observation}
  In particular, there is no partition of $A_k$ into two contiguous partition classes such that both have boundary size less than $6k$.  With this motivation, we define $\Omega=\Omega_{C,\varepsilon,k}$ to be the partitions of $A_k$ into two contiguous pieces, each with boundary sizes at most $6k+Ck^{1-\varepsilon}$, and consider the problem of uniform sampling from $\Omega$.    We will show that this problem can be solved in polynomial time with our approach, but also that Glauber dynamics on this state space has exponential mixing time. Observe that we can equivalently view $\Omega$ as set of paths in $A_k'$ between points of $\partial A_k'$, and for any partition $\omega\in \Omega$ we write $\PP_\omega$ for this corresponding path.

  Writing $\omega\sim \omega'$ for $\omega,\omega'\in \Omega$ whenever (viewed as partitions) $\omega,\omega'$ agree except on a single vertex of $A_k$, we define the Glauber dynamics for $\Omega$ to be the Markov chain which transitions from $\omega$ to a uniformly randomly chosen neighbor $\omega'$.
  Recall that we define the \emph{conductance} by
  \begin{equation}
\Phi=\min_{\pi(S)\leq \frac 1 2}\frac{Q(S,\bar S)}{\pi(S)}
  \end{equation}
  where
  \[
  Q(S,\bar S)=\sum_{\substack{\omega\in S\\\omega'\in \bar S}}\pi(\omega)P(\omega,\omega')\leq \pi(\partial S),
  \]
  where $\partial S$ is the set of all $\omega\in S$ for which there exists an $\omega'\in \bar S$ for which $P(\omega,\omega')>0$.

  The mixing time $t_{\mathrm{mix}}$ of the Markov chain with transition matrix $P$ is defined as the minimum $t$ such that the total variation distance between $vP^t$ and the stationary distribution $\pi$ is $\leq \frac 1 4$, for all initial probability vectors $v$.  With these definitions we have
  \begin{equation}
     t_{\mathrm{mix}}\geq \frac{1}{4\Phi}
  \end{equation}
  (e.g. see \cite{Peres}, Chapter 7) and so to show the mixing time is exponentially large it suffices to show that the conductance $\Phi$ is exponentially small.

  To this end, we define $S\subseteq \Omega$ to be the set of $\omega$ for which the endpoints $(x_1,y_1)$ and $(x_2,y_2)$ of $\PP_\omega$ satisfy
  \begin{equation}
    x_1\leq x_2\quad y_1\leq y_2.
  \end{equation}

  Our goal is now to show that $|S|$ is large while $|\partial S|$ is small.  For simplicity we consider the case where $k$ is even but the odd case can be analyzed similarly.

  To bound $S$ from below it will suffice to consider just the partitions whose boundary path in $A_k'$ is a shortest path from the point $(-\tfrac k 2,-\tfrac k 2)$ to the point $(\tfrac k 2,\tfrac k 2)$; note that such a path for the case where $k=4$ is shown in Figure \ref{f.aztec}  There are $\binom{2k}{k}$ such paths and so we have lower bound
  \begin{equation}\label{Slower}
|S|\geq \binom{2k}{k}=\Omega\left(\frac{2^{2k}}{\sqrt k}\right).
  \end{equation}

  To bound $|\partial S|$ from above We will make use of the following count of walks in the lattice:
  \begin{lemma}
	\label{l:number_of_walks}
	For any point $P = (n_1, n_2)$ such that $n_1 + n_2 = n$, the number of walks from $O = (0,0)$ to $P$ of length $n + 2t$ is given by
	\[ \binom{n + 2t}{t} \binom{n+2t}{n_1 + t} \]
\end{lemma}
\begin{proof}
	Let $W_t$ denote number of such walks. Note that any such path can be denoted as a sequence of symbols $U,D,L,R$ which denote moves in the corresponding directions. For a direction $Z \in \set{U,D,L,R}$, let $n_Z$ denote number of symbols signifying the direction that appear in the walk; then the walks from $O$ to $P$ are in bijection with the sequences over $\{U,D,L,R\}$ of length $n+2t$ for which $n_U - n_D = n_1$ and $n_R - n_L = n_2$.  Note then that $n_L+n_D=t$, $n_U+n_L=n_1+t$, and $n_R+n_D=n_2+t$.  There is a bijection from the set of these sequences $s$ to pairs of subsets $(X_s,Y_s)\subseteq [n+2t]$ where $|X_s|=t$ and $|Y_s|=n_1+t$ as follows.  Given such a sequence $s$, we can let $X_s$ be the set of indices with symbols $L$ or $D$, while $Y_s$ is the set of indices with symbols $U$ or $L$.  The sequence $s$ is recovered from the sets $X_s$ and $Y_s$ by assigning the symbol $U$ to indices in $X_s\setminus Y_s$, the symbol $L$ to indices in $X_s\cap Y_s$, the symbol $D$ to those in $Y_s\setminus X_s$, and the symbol $R$ to indices in neither $X_s$ nor $Y_s$.
\end{proof}
Now the boundary $\partial S$ of $S$ thus consists of paths which satisfy either $x_1=x_2$ or $y_1=y_2$.  Observation \ref{aligned}, together with the condition that the total length of a closed walk enclosing each partition class is at most $6k+O(k^{1-\varepsilon})$, implies that in these cases, we must have $|y_i|=O(k^{1-\varepsilon})$ in the case where $x_1=x_2$ or $|x_i|=O(k^{1-\varepsilon})$ in the case where $y_1=y_2$.  In particular, we have without loss of generality that $x_1=x_2$, and $y_2=y_1+2k-O(k^{1-\varepsilon})$.  In particular, letting $\ell_\omega=y_2-y_1$, we have that the path $\PP_\omega$ has length $\ell_\omega + O(\ell_\omega^{1-\varepsilon})$. Now by Lemma \ref{l:number_of_walks}, the number of choices for such walks (for fixed $x_i,y_i$, for which there are only polynomially many choices) is
\asnote{changing $\ell$ to $\ell_\omega$}
\begin{equation}\label{Supper}
\binom{\ell_\omega+O(\ell_\omega^{1-\varepsilon})}{O(\ell_\omega^{1-\varepsilon})}^2\leq 2^{O(\ell_\omega^{1-\varepsilon})}
\end{equation}
for $0\leq \varepsilon\leq 1$. Together, \eqref{Supper} and \eqref{Slower} imply that
\[
\Phi= \frac{2^{O(\ell_\omega^{1-\varepsilon})}}{2^{2k}/\sqrt{k}}\lesssim \frac{1/4}{2^{\varepsilon k}},
\]
and so the mixing time $t_{\mathrm{mix}}$ satisfies
\begin{equation}
  t_{\mathrm{mix}}\geq 2^{\varepsilon k},
\end{equation}
with respect to the fixed parameter $\varepsilon>0$. This gives the following theorem:
\begin{theorem}[\Cref{t:glauber_mixing} restated]
    Glauber Dynamics on contiguous $2$-partitions of $A_k$ with boundary of length at most $6k + Ck^{(1-\epsilon)}$ has exponential mixing time.
\end{theorem}

On the other hand, we claim that we can sample the partitions $\omega \in \Omega$ efficiently using \Cref{alg:sampling_paths}, by applying it to each pair of points on the boundary $A_k'$, to generate the path $P_\omega$. To show this, we will argue that the set $A_k'$ has the correct width property with endpoints of $P_\omega$. Formally,
\begin{lemma}
    Let $\omega \in \Omega$ be a partition of $A_k$. Let $P_\omega$ be corresponding path in $A_k'$ with endpoints $P_1,P_2$. Then $A_k'$ is $(\ell, 16 \ell + 4Ck^{(1-\epsilon)}, 0)$-wide with respect to points $P_1,P_2$ for all $\ell$.
\end{lemma}
\begin{proof}
    Let $P_i = (x_i,y_i)$ for $i = 1,2$. Without loss of generality, let $(x_2, y_2) \ge (0,0)$. Let $Q = (-x_2,-y_2)$ be the point anti-podal to $P_2$ in $\partial A_k'$. We will break the proof into three cases, based on which quadrant $P_1$ is in.
    
    Suppose $P_1$ is in third quadrant. Then the distance between $P_1$ and $P_2$ is exactly $2k$. Therefore, $P_2$ is at most $Ck^{(1-\epsilon)}$ distance from $Q$. The lattice box $\mathcal{R}(P_1,P_2)$ has at most 
    \[ 2 \abs{x_1 + x_2} + 2 \abs{y_1 + y_2} \]
    points in $\partial A_k'$. This is exactly the distance between $P_1$ and $Q$. Therefore, a shortest path from $P_1$ to $P_2$ can intersect $\partial A_k'$ at at most $2Ck^{(1-\epsilon)}$ many points. It follows that a path of length $2k + 2\ell$ is contained in $\mathcal{R}(P_1-(\ell,\ell), P_2+(\ell,\ell))$, which contains at most $16 \ell + 2Ck^{(1-\epsilon)}$ points in $\partial A_k'$, implying that any path of length $2k + 2\ell$ can intersect $\partial A_k'$ in at most $16 \ell + 4 C k^{(1-\epsilon)}$.
    
    Suppose $P_1$ is in the second quadrant. Then the distance between $P_1$ and $P_2$ is
    \[ x_2 - x_1 + \abs{y_2 - y_1} = x_2 - x_1 + \max(y_1, y_2) - \min(y_1, y_2) \ge 2k - 2 \min(y_1,y_2) \]
	Further, length of the lower boundary of $A_k$ between $P_1$ and $P_2$ is at least $4k + y_1 + y_2$, and hence boundary of the lower partition is at least $6k + \abs{y_2 - y_1}$, which implies that
	\[ \abs{y_2 - y_1} \le  Ck^{(1-\epsilon)} \]
	The lattice box $\mathcal{R}(P_1,P_2)$ contains at most $2\abs{y_2 - y_1} + 4$ points on the boundary $\partial A_k'$. By similar argument to above, we can conclude that any path of length $2\ell$ larger than the shortest path is contained in a slightly bigger lattice box, and can intersect the boundary $\partial A_k'$ in at most
	\[ 2\abs{y_2 - y_1} + 16\ell + 4 \le 16\ell + 4Ck^{(1-\epsilon)} \]
	points.

	The case when $P_1$ is in the fourth quadrant is handled similarly to the case when $P_1$ is in the second quadrant. This proves that in all cases, the Aztec Diamond is $(\ell, 16\ell + 4Ck^{(1-\epsilon)}, 0)$-wide. 
\end{proof}

This lemma implies that for $\ell \le Ck^{(1-\epsilon)}$, and $s(l) = 20Ck^{(1-\epsilon)}$, the set $A_k'$ satisfies the hypothesis of \Cref{t:long_cycle_bound_subgraph} for all points $P_1,P_2$ that are endpoints of $P_\omega$ for some $\omega \in \Omega$. Hence, for each pair of points $P_1, P_2 \in \partial A_k'$, we can compute $W_\ell^\lambda(P_1,P_2)$ for all $\ell \le Ck^{(1-\epsilon)}$, where $\lambda \epsilon > 1$. This allows us to uniformly sample $P_\omega$, for $\omega \in \Omega$, with rejection sampling, using the following algorithm:

\begin{algorithm}[H]
	\caption{Partition Sampling}
	\label{alg:aztec_sampling}
	\begin{algorithmic}[1]
		\State Compute $\Call{DP}{Q,P,w,t}$ for all $Q \in A_k'$, $P \in \partial A_k'$, $w \in \Phi_\lambda$, $0 \le t \le 2k + Ck^{1-\varepsilon}$ using \Cref{alg:low_girth_walks}
		\While{$P_\omega$ is not a path}
			\State Sample $P_1,P_2,t$ proportional to $\Call{DP}{Q,P,\{O\},t}$
			\State Sample $P_\omega$ from $P_1$ to $P_2$ of length $t$ using \Cref{alg:sampling_walks}
		\EndWhile
		\State \Return $P_\omega$
	\end{algorithmic}
\end{algorithm}

\newpage
\nocite{*}
\bibliographystyle{plainnat}
\bibliography{main}

\newpage
\appendix
\section{Bounds on Binomial Coefficients}
\label{s:binomial_bounds}
We first recall some exponential bounds on $1+x$. We have the standard upper bound:
\begin{equation}
	\label{eq:exp_inequality_1}
	e^x \ge 1 + x \qquad \forall x \in \R
\end{equation}
On the other hand, we have the lower bound:
\begin{equation}
  \label{eq:exp_inequality_2}
	e^{\frac{x}{1+x}} \le 1 + x \le e^x \qquad \forall \, x > - 1
\end{equation}
This follows since 
\[ 1 - t \le e^{-t} \implies 1 - \frac{x}{1+x} \le e^{-\frac{x}{1+x}} \implies \frac{1}{1+x} \le e^{-\frac{x}{1+x}} \]
We get \Cref{eq:exp_inequality_2} from this by taking reciprocals whenever $\tfrac{1}{1+x} \ge 0$. Further, \Cref{eq:exp_inequality_2} implies that
\begin{equation}
	\label{eq:exp_inequality_3}
	e^{\frac{x}{2}} \le 1 + x \le e^x \qquad \forall \, 0 \le x \le 1
\end{equation}
We also recall the Sterling's Approximation - the non-asymptotic version of Sterling's Approximation is given in \citet{robbins1995stirling} as
\begin{equation}
	\label{eq:sterling_approximation} 
	\sqrt{2 \pi n} \brac{\frac{n}{e}}^n \exp\brac{\frac{1}{12n + 1}} \le n! \le \sqrt{2 \pi n} \brac{\frac{n}{e}}^n \exp\brac{\frac{1}{12n}}
\end{equation}

We can use these exponential bounds on $(1+x)$ to bound the binomial coefficients. In particular, we are interested in bounded the binomial coefficient $\binom{n+x}{k}$ in the case where $x,k \le \tfrac{n}{10}$. Recall by the definition of binomial coefficients:
	\[
		\binom{n+x}{k}
	= \frac{1}{k!} \prod_{i=0}^{k-1} \brac{n + x - i}
	= \frac{n^k}{k!} \prod_{i=0}^{k-1} \brac{1 + \frac{x-i}{n}}
\]
Using \Cref{eq:exp_inequality_1} we get the following upper bound:
\begin{align*}
	\binom{n+x}{k}
	& \le \frac{n^k}{k!} \exp\brac{\sum_{i=0}^{k-1} \frac{x-i}{n}} \\
	& \le \frac{n^k}{k!} \exp\brac{\frac{2kx - k^2 + k}{2n}}
\end{align*}
Using \Cref{eq:exp_inequality_3} we get the following lower bound when $n + x - k \ge \abs{x},k$:
\begin{align*}
	\binom{n+x}{k}
	& \ge \frac{n^k}{k!} \exp\brac{\sum_{i=0}^{k-1} \frac{\frac{x-i}{n}}{1 + \frac{x-i}{n}}} \\
	& = \frac{n^k}{k!} \exp\brac{\sum_{i=0}^{k-1} \frac{x-i}{n+x-i}} \\
	& = \frac{n^k}{k!} \exp\brac{\sum_{i=0}^{k-1} \frac{x-i}{n} + \frac{x-i}{n+x-i} - \frac{x-i}{n}} \\
	& = \frac{n^k}{k!} \exp\brac{\sum_{i=0}^{k-1} \frac{x-i}{n} - \frac{(x-i)^2}{n(n+x-i)}} \\
	& \ge \frac{n^k}{k!} \exp\brac{\frac{2kx - k^2 + k}{2n} - \frac{2k(\abs{x} + k)}{n}}
\end{align*}
Where the last inequality follows since $(x-i)^2 \le 2x^2 + 2 i^2 \le 2x^2 + 2k^2 \le 2(\abs{x} + k)(n+x-i)$ assuming that $n + x - i \ge \abs{x},k$. Together, we get the following upper and lower bounds on the binomial coefficients:
\begin{equation}
	\label{eq:binomial_bounds}
	\frac{n^k}{k!} \exp\brac{\frac{2kx - k^2 + k}{2n} - \frac{2k\abs{x} + 2k^2}{n}} \le \binom{n+x}{k} \le \frac{n^k}{k!} \exp\brac{\frac{2kx - k^2 + k}{n}}
\end{equation}

\end{document}